\documentclass[12pt,a4paper,reqno]{amsart}

\usepackage[T1]{fontenc}
\usepackage[utf8]{inputenc}
\usepackage{lmodern}
\usepackage[english]{babel}
\usepackage[activate]{pdfcprot}
\usepackage{enumerate}
\usepackage{amssymb}
\usepackage{url}

\usepackage{verbatim} %for comment environment
\usepackage{marginnote}

\usepackage{datetime}
\usepackage{soul}
\usepackage{amssymb}
\usepackage{amscd}
\usepackage{amsfonts}
\usepackage{mathtools}
\usepackage{enumitem}
\setlist{beginpenalty=100}

\newtheorem{thm}{Theorem}[section]
\newtheorem{lem}[thm]{Lemma}
\newtheorem{cor}[thm]{Corollary}
\newtheorem{prop}[thm]{Proposition}

\theoremstyle{definition}
\newtheorem{de}[thm]{Definition}

\theoremstyle{remark}
\newtheorem{rem}[thm]{Remark}

\numberwithin{equation}{section}

\newcommand{\rmnum}[1]{\romannumeral #1}
\newcommand{\Rmnum}[1]{\expandafter\@slowromancap\romannumeral #1@}

\newcommand{\be}{\mathbf{e}}
\newcommand{\bN}{\mathcal{N}}

\newcommand{\bv}{\mathbf{v}}
\newcommand{\bw}{\mathbf{w}}
\newcommand{\bx}{\mathbf{x}}

\newcommand{\R}{\mathbb{R}}

\newcommand{\K}{\mathbb{K}}
\DeclareMathOperator{\SL}{SL}

\DeclareMathOperator{\GL}{GL}

\newcommand{\Z}{\mathbb{Z}}
\newcommand{\Q}{\mathbb{Q}}

\newcommand{\F}{\mathbb{F}}

\DeclareMathOperator{\diag}{diag}

\DeclareMathOperator{\Stab}{Stab}

\DeclarePairedDelimiterX\set[1]{\lbrace}{\rbrace}{%

#1
}
\usepackage{xcolor}

\newcommand{\red}[1]{{\color{red}#1}}

\allowdisplaybreaks

\begin{document}

\title{The Product of linear forms over function fields}

\author{Wenyu Guo}
\address{Shanghai Center for Mathematical Sciences, Jiangwan Campus, Fudan University, No.2005 Songhu Road, Shanghai, 200438, China}
\email{wyguo23@m.fudan.edu.cn}

\author{Xuan Liu}

\address{Shanghai Center for Mathematical Sciences, Jiangwan Campus, Fudan University, No.2005 Songhu Road, Shanghai, 200438, China}
\email{24110840009@m.fudan.edu.cn}

\author{Ronggang Shi}
% Address of record for the research reported here
\address{Shanghai Center for Mathematical Sciences, Jiangwan Campus, Fudan University, No.2005 Songhu Road, Shanghai, 200438, China}
% Current address
%\curraddr{Department of Mathematics, Ohio State
 %University, Columbus, Ohio 43210}
 \email{ronggang@fudan.edu.cn}
% \thanks will become a 1st page footnote.

\thanks{The author is supported by  NSFC 12161141014, NSF Shanghai 22ZR1406200
and the New Cornerstone Science Foundation.}

% General info


%\date{\usdate\today}
%No need date for arxiv.

\begin{abstract}
The aim of this paper is to study the product of $n$ linear forms over function fields. We calculate  the maximum value of the minima of the forms with determinant one when
$n$ is  small. The value is  equal to  the natural  bound given by algebraic number theory.
Our proof is based on a reduction theory of 
 diagonal group orbits  on homogeneous spaces.
We also  show that the forms defined algebraically correspond
  to periodic orbits  with respect to the diagonal group actions. 
\end{abstract}

\maketitle

%\tableofcontents

\section{Introduction}\label{sec;introduction}

The study  of the  product of  linear forms with real coefficients  has a long history, see e.g.~\cite{Davenport38}\cite{Rogers50}\cite{Cassels55}.  Some challenging open questions remain unanswered, such as Littlewood conjecture and Cassels-Swinnerton-Dyer-Margulis conjecture.  
In recent years the dynamics of diagonal group actions on homogeneous spaces plays an important role, see e.g.~\cite{LW} and \cite{EKL}. 

The aim of this paper is to study the product of linear forms with coefficients in 
the  field $\K=\F((t^{-1}))$ of formal Laurent series over a finite field $\F$. 
These forms are related  to the diagonal group actions on the homogeneous space $G/\Gamma$ where  $G=\SL_n(\K)$ and $\Gamma=\SL_n(\F[t])$.
Besides its own interests, we hope to give some insights to the classical questions.

Suppose the field $\F$  has $q=p^e$ elements where $p$ is a prime number
and $e$ is a positive integer.
 We endow  $\K$ with an ultrametric norm $|\cdot|$ so that  
$|bt^n|=q^n$ for $b\in \F\setminus\{ 0\}, n\in \Z$.  It is  the completion of the  rational function field $\F(t)$ at $\infty$. 

Let $L(\bx)=l_1(\bx) \cdots l_n(\bx)$ be a nondegenerate product of linear forms in $n$-variables ($n\ge 2$), i.e.~$l_1, \ldots, l_n$ are linearly independent linear forms on $\K^n$.  It is clear that $L$ is nondegenerate if and only if it has nonzero determinant $\det(L)$ which is the determinant of the matrix $g=(a_{ij})$ if $l_i(\bx)=a_{i1}x_1+ \cdots + a_{in}x_n$. 
In this way,  an $n$-by-$n$ matrix $g$ gives a product of linear forms $L_g$. This gives a parameterization of products of linear forms by matrices. 

We are trying to understand the values of $L$ at $\bx\in \F[t]^n$. 
Since  $(aL)(\bx)=a\cdot L(\bx) $ for $a\in \K$, it suffices to understand the values of  forms which have determinant $1$. 
These forms are parameterized by elements of $G$.
The minimal norm of values of $L$ is denoted by
\[
\bN (L)=\inf  \{ |L(\bx)|: \bx\in \F[t]^n, \bx \neq 0   \}.
\]
Let  $\mathfrak m_n$ be the infimum of the real numbers in the set
\[
\mathfrak M_n=\{ \bN(L_g)^{-1}:  g\in G\}=\{ \bN(L_g)^{-1} |\det (g)|:  g\in \GL_n(\K)\}.
\]
In this paper we  compute $\mathfrak m_n$ when $\frac{n}{ q}$ is small. 
\begin{thm}\label{thm;main}
We have $\mathfrak m_n= q^{n-1}$ when $2\le n\le q+1$ and $\mathfrak m_n=q^ n$ when $q+2\le n\le N_q$ where 
\begin{align*}
    N_q= \begin{cases}
    4 & \text{if } q=2,\\
    q+\lfloor2\sqrt{q}\rfloor & \text{if  $p|\lfloor2\sqrt{q}\rfloor$, $e$ is odd  and $e\ge 5$}, \\
    q+\lfloor2\sqrt{q} \rfloor +1& \text{otherwise}.
\end{cases}
\end{align*}
\end{thm}

Here $\lfloor r\rfloor$ is the smallest integer less than or equal to $r$. 
The value of $\mathfrak m_n$ can be interpreted algebraically as the minimal absolute value of 
discriminants of certain degree $n$ extensions of $\F(t)$.
Recall that for a finite separable extension $k$ of 
degree $n$ over $\F(t)$ the discriminant of $k/\F(t)$ is 
\[
d_k=\det (\mathrm{Tr}_{k/\F(t)}(\alpha_i\alpha_j))
\]
where $\alpha_1,\ldots , \alpha_n$ is an $\F[t]$-basis of the integral closure of $\F[t]$ in $k$. 
We remark here that $d_k$ is well-defined up to multiplications by 
elements of  $\{c^2: 0\neq c\in \F\}$. This ambigurity will not affect its 
usage and is more convenient for our purposes. 
Moreover, we say $k$ is split at $\infty$ if $k=\F(t, \alpha)$ and the minimal polynomial of $\alpha$ over $\F(t)$ has $n$-distinct roots in $\K$. 
This is equivalent to  that 
there are $n$-distinct $\F(t)$-embeddings of $k$ into $\K$, say $\sigma_1, \ldots, 
\sigma_n$.  For the 
\begin{align}\label{eq;rational}
L(\bx)=\prod_{i=1}^n \sigma_i(\alpha_1) x_1+\sigma_i(\alpha_2)x_2+\cdots+ \sigma_i(\alpha_n) x_n,
\end{align}
one has $\bN (L)=1$ and $\det (L)= |d_k|^{1/2}$. 
Therefore, $|d_k|^{1/2}\in  \mathfrak M_n$ for all degree $n$ extensions $k$ of $\F(t)$ split at $\infty$. We use $\textup{disc-}\mathfrak{m}_n$ to denote the minimum of these $|d_k|^{1/2}$. Since $\mathfrak m_n $ is the infimum of values in $\mathfrak M_n$, we have 
\begin{align}\label{eq;upper}
\mathfrak m_n \le \textup{disc-}\mathfrak{m}_n.
\end{align}
For those integers $n$ in Theorem \ref{thm;main} we actually have equality.
\begin{thm}\label{thm;compare}
For $2\le n\le N_q$ we have $\mathfrak m_n = \textup{disc-}\mathfrak{m}_n$.
\end{thm}

In the classical case $\F[t], \F(t)$ and $ \F((t^{-1}))$  corresponds to  $\Z, \Q$ and $\R$, respectively. A field $k$ split at infinity corresponds to  a totally real extension of $\Q$. 
In the classical case it is known that $\mathfrak m_2 =\sqrt 5=\textup{disc-}\mathfrak{m}_2$ by Hurwitz \cite{Hurwitz} and
$\mathfrak m_3 =7=\textup{disc-}\mathfrak{m}_3$ by Davenport \cite{Davenport38}. Even for $n=4$, the precise value of 
$\mathfrak m_4 $ is not known. Theorem \ref{thm;compare} gives more evidence that equality might  hold in (\ref{eq;upper}) for both real and function field cases.

We estimate $\mathfrak m_n$ from below and above  by different methods. For the lower bound we
develop a reduction theory  for $A$-orbits of $G/\Gamma$, 
where 
$A$ be subgroup of $G$ consisting of diagonal matrices. 
We state the reduction as a decomposition of  the group  $G$. 
Let $\mathfrak{p}=\{\alpha\in\K: |\alpha|<1\}$ be the 
unique maximal ideal of the  valuation ring   $\mathfrak{o}$ $=\{\alpha\in\K:|\alpha|\leq 1\}$ of  $\K$. 

\begin{thm}\label{thm:AKT}
We have 	$G=A\SL_n(\mathfrak o;\mathfrak p)\Gamma$, where $\SL_n(\mathfrak o;\mathfrak p)=\{g\in \SL_n(\mathfrak o): g\equiv  1_n \mod \mathfrak p \} $.
\end{thm}

The reduction relies on the ultrametric norm  and the principal ideal structure of $\F[t]$.
Theorem \ref{thm:AKT} refines \cite[Theorem 1.24]{Aranov}. The analogue of Theorem \ref{thm:AKT} in the real case is called DOTU-decomposition and it does not hold when  $n\ge 64$, see \cite[p.~624]{gruber} and the references there.
We use Theorem \ref{thm:AKT}  to give 
 a lower bound of  $\mathfrak m_n$ for all $n$ and estimate asymptotically from below. 
\begin{thm}\label{thm;mainlow}
$\log _q \mathfrak m_n\ge \max \{n-2+  \frac{n}{q+1},  \frac{nq-3q
\log _q n }{q-1} \}$. In particular, we have 
\[
    \liminf_{n\rightarrow\infty} \frac{\log_{q}\mathfrak{m}_n}{n}\geq \frac{q}{q-1}.
    \]
\end{thm}
% The statement may needs to be modified, but the goal is to calculate it for all $n$.

For the upper bound we estimate $\textup{disc-}\mathfrak m_n$ using algebraic number theory and arithmetic geometry of curves over finite fields. After reviewing some basic facts of function fields, we
reduce the calculation of $\textup{disc-}\mathfrak m_n$ to that of the minimal genus of degree
$n$ extensions of $\F(t)$ split at $\infty $. Then we calculate precisely when the minimal genus is equal to $0$ or $1$ using  results on counting rational points of curves. This is the main reason that we have precise values of  $\mathfrak m_n$ for small $n$. 

\begin{thm}\label{thm;algebraic}
  We have  $
 \textup{disc-}\mathfrak{m}_n=\begin{cases}
  q^{n-1} & 2\le n\le q+1\\
  q^n &  q+2 \le n \le N_q
 \end{cases}
$ and $\textup{disc-}\mathfrak{m}_n\ge q^{n+1} $  for $n>N_q$.
\end{thm}

It is possible to give an  upper bound of $\text{disc-}\mathfrak m_n$ by explicitly constructing extensions using class filed theory. We will not study this direction in details to avoid technical complexity.

We say $L$ has bounded type if 
$\bN (L)>0$.
If 
$L$ is proportional to a form  in (\ref{eq;rational}) with $\alpha_1, \ldots, \alpha_n$
an $\F(t)$-basis of $k$, then $L$ has bounded type. For $n\ge 3$, it is an open question whether the converse holds. 
\begin{comment}
We say $L$ is rational if $L$ is  proportional to a form with coefficients in $ \F(t)$. 
We say $L$ does not represent
zero if $L(\bx)\neq 0$ for all nonzero $\bx\in \F[t]^n$. 
Clearly, a form  $L$  proportional to (\ref{eq;rational}) is rational and 
and does not represent
zero.
\end{comment}
The forms in (\ref{eq;rational}) also  correspond to 
periodic $A$-orbits 
in the finite volume homogeneous space $G/\Gamma$. Here an orbit $Ag\Gamma$ is periodic 
in $G/\Gamma$ means $\textup{Stab}_G(g\Gamma)$ is a lattice of $A$, which is equivalent 
to saying that $Ag\Gamma$ is compact. 
%We sum up these characterizations in the following theorem.

\begin{thm}\label{thm;equivalent}
Suppose $g\in G$. The followings are equivalent:
\begin{enumerate}[label=\textup{(\roman*)}]
  %  \item $\bN(L_g)>0$ and $\{ N(g \bx): \bx\in \F[t] \}$ is discrete in $\K$.
    \item $L_g$ is proportional  to (\ref{eq;rational}) for a basis
    $\alpha_1, \ldots, \alpha_n$ of  a degree $n$ extension $k/\F(t)$ split at $\infty$. 
 %  \item $L_g$ is rational and does not represent zero.
    \item $Ag\Gamma$ is periodic  in $G/\Gamma$. 
\end{enumerate}
\end{thm}

The real  case of Theorem \ref{thm;equivalent}
is well-known and our  proof is the same.
%We give detailed proofs in \S \ref{section;periodic} for completeness.
Theorem \ref{thm;equivalent} allows one  to use homogeneous dynamics to 
study isolation of rational forms as in \cite{LW}.
Other related questions 
  were studied recently by several authors, see e.g.~\cite{Erez}\cite{Aranov}\cite{Paulin}\cite{Bagshaw-Kerr}.

\subsection*{Acknowledgments}
The authors are thankful to Noy Aranov for carefully reading an initial version of this article as well as for helpful bibliographical suggestions.

\section{Geometry of numbers over function fields}
We  develop a  reduction theory in Proposition \ref{lem;good basis}
which gives the decomposition of $G$ in Theorem \ref{thm:AKT}.
This allows us to give a lower bound of $\mathfrak m_n$. 
We write $R=\F[t]$ to simplify the notation and fix $n\ge 2$.

We consider $\K^n$ as column vectors over $\K$ and endow a norm
$\|\cdot\|$ on it defined by 
\[
		\Vert (v_1,v_2,\ldots,v_n)^{\mathrm{tr}}\Vert= \max\{|v_1|, |v_2|,\ldots,|v_n|\}. 
	\] 
The geometry of numbers over function fields considers discrete and cocompact $R$-modules $\Lambda$ in $\K^n$.  Since $R$ is a principle ideal domain we know that all these $\Lambda$ are free $R$-modules of rank $n$. Let $B_r$ be the closed ball of radius $r$ in $\K^n$, i.e.
\[
B_r=\{\bv\in \K^n : \|\bv\|\le  r \}.
\]
For $1\le i\le n$, the $i$-th minimum of $\Lambda$ is
	\[
		\lambda_i(\Lambda) =\inf\{r>0:  \dim\mathrm{span}_\K (\Lambda \cap B_r)\ge i \}.
	\]

Let  $\Lambda$  be a discrete and cocompact  $R$-module in $\K^n$.
We fix  $\Lambda $ and  write $\lambda_i$ for $\lambda_i(\Lambda)$ to simplify the notation. We also take $\lambda_0=0$ for convenience. 
We first prove several auxiliary lemmas of Proposition \ref{lem;good basis}. 
These lemmas are standard (see e.g. \cite{Mahler} or \cite{Aranov}) and  we give proof to make the paper self-contained. 
The following lemma follows directly from the definition of successive minima and will be used several times in our discussions below. 

\begin{lem}\label{lem;rank}
 Suppose $1\le j\le n$ and $\lambda_{j-1}\le r<\lambda_j$. Then 
\[
\mathrm{span}_\K (\Lambda \cap B_r)= \mathrm{span}_\K (\Lambda \cap B_{\lambda_{j-1}}).
\]
In particular, the two spaces have the same dimension.
\end{lem}

\begin{lem}\label{lem;unique}
   Let  $\Lambda$  be a discrete and cocompact  $R$-module in $\K^n$. 
Let  $\bv_1,\ldots,\bv_n\in \Lambda$ be $\K$-linearly independent vectors  such that  $\|\bv_{i}\|\le \lambda_{i}$.  Then we have $\|\bv_{i}\|= \lambda_{i}$ for every $1\le i \le n$. 
\end{lem}
\begin{proof}
    Suppose there exists some $i$ with $\|\bv_{i}\|< \lambda_{i}$. Let $\ell$ be  the smallest such $i$. Let $j$ be the smallest $i$ such that $\lambda_i=\lambda_\ell$.  
    We choose $r$ with  $\|\bv_\ell\|<r<\lambda_j$ and $r>\lambda_{{j-1}}$. Then 
$$\bv_1, \ldots, \bv_{j-1}, \bv_{\ell}\in B_r \cap \Lambda,$$
which contradicts Lemma \ref{lem;rank}.    
\end{proof}
 
From the general theory  of geometry of numbers (see e.g. \cite{kst}) we know that there are vectors $\bv_1, \ldots, \bv_n$ of $\Lambda$ with $\|\bv_i\|=\lambda_i$ such that 
the index  of $R$-span of $\{ \bv_1, \ldots, \bv_n \}$ in $\Lambda$ is bounded by an absolute constant independent of $\Lambda$. 
Here we strengthen this result.

\begin{lem}\cite[Lemma 1.18]{Aranov}\label{lem;Succesive minima}
%	Let  $\Lambda$  be a discrete and cocompact  $R$-module in $\K^n$.   
Let  $\bv_1,\ldots,\bv_n\in \Lambda$ be $\K$-linearly independent vectors  such that  $\|\bv_{i}\|=\lambda_{i}$. 
 Then the 
  $R$-linear span of  $ \bv_1, \ldots, \bv_n $ is $  \Lambda$.
  %such that $|\bv_i|=\lambda_i $ for $1\leq i \leq n$.
\end{lem}
\begin{rem}    
    Following \cite{Aranov},  a basis of $\Lambda$ satisfying  Lemma \ref{lem;Succesive minima} is called  a basis of successive minima for $\Lambda$.
\end{rem}
\begin{proof}

%Since $\Lambda$ is discrete and cocompact, 
%there exists The key point here is to show that 
Suppose the contrary, then there exists  $\bw \in \Lambda$ not in 
$\mathrm{span}_R \{\bv_1, \ldots, \bv_n \}$. 
 There are uniquely determined $a_1,\ldots, a_n\in \K$ not all in $R$ such that 
	\[
		\bw=a_1\bv_1+\cdots+a_n\bv_n.
	\]
 We can write $a_i= b_i+c_i$ for some uniquely determined $b_i\in R$ and $c_i\in t^{-1}\F[[t^{-1}]]$. In particular, we have $|c_i|<1$. Let $1\le \ell\le n$ be the largest number such that $c_{\ell}\neq 0$.
Then
	\[
		\bw'={c_1}\bv_1+\cdots+{c_{\ell}}\bv_{\ell} \in \Lambda\setminus \mathrm{span}_\K \{ \bv_1, \ldots, \bv_{{\ell}-1} \}.
	\]
Since the norm on $\K^{n}$ is ultrametric we have 
	\[
		\Vert \bw' \Vert \leq \max\{\Vert{c_1}\bv_1\Vert,\cdots,\Vert{c_{\ell}}\bv_{\ell}\Vert\}<\Vert\bv_{{\ell}}\Vert=\lambda_{{\ell}}.\nonumber
	\]
We claim that this is impossible. Let  
\[
j=\min\{i\le {\ell}: \lambda_i =\lambda_{\ell}\}. 
\] We choose $r$ with  $\|\bw'\|<r<\lambda_j$ and $r>\lambda_{{j-1}}$. Then 
$$\bv_1, \ldots, \bv_{j-1}, \bw'\in B_r \cap \Lambda,$$
which contradicts Lemma \ref{lem;rank}.
% the definition of $\lambda_j$. 
\end{proof}

% We will further refine a given  basis of successive minima  for $\Lambda$ so that each vector has only one component with maximal absolute value.
We will further refine a given  basis of successive minima by linear operations. 
 For 
 $b\in \K$ we define its 
 %leading term to be $\ell(b)= \theta  t^m$ where $0\neq \theta \in \F $
 %and $m\in \Z$ so that $|b-\theta t^m|< |b|$. If $b=0$ we set $\ell (b)=0$ for convenience. The 
 integral part to be the unique element $[b]\in \F[t]$ such that $|b-[b]|< 1$.

 \begin{lem}\label{lem;basic}
 %Let  $\Lambda$  be a discrete and cocompact  $R$-module in $\K^n$. 
Let $\bw_j=(b_{1j}, \ldots, b_{nj})^{\mathrm{tr}} \  (1\le j\le n)$ be a basis of successive minima for $\Lambda$. Fix ${\ell}\le n$ and choose $i$ so that 
$\|\bw_{\ell} \|= |b_{i{\ell}}|$. Given $s\neq {\ell}$ we define a new basis $\bv_1, \ldots, \bv_n$ of $\Lambda$  as follows: $\bv_j=\bw_j$ if $j\neq s$ and\[
\bv_s= \bw_s-[b_{is}b_{i{\ell}}^{-1}]\bw_{\ell}.
\] 
Then $\bv_1, \ldots, \bv_n$ is a basis of successive minima for $\Lambda$.
 \end{lem}
\begin{proof}
In view of Lemma \ref{lem;Succesive minima},
it suffices to show $\|\bv_s\|=\lambda_s$ for the given  $s$. 
   If $\lambda_s<\lambda_{\ell}$, then \[
|b_{is}|\le \|\bw_s\|=\lambda_s<\lambda_{\ell}=|b_{i{\ell}}|
   \]
   In this case, $[b_{is}b_{i{\ell}}^{-1}]=0$ and $\bv_s=\bw_s$. So $\|\bv_s\|=\lambda_s$ holds trivially. 

   Suppose $\lambda_s\ge \lambda_{\ell}$. We have 
   \[
   \|[b_{is}b_{i{\ell}}^{-1}]\bw_{\ell}\|\le |b_{is}|\cdot |b_{i{\ell}}^{-1}|\cdot \|\bw_{\ell}\|\le \lambda_s \lambda_{\ell}^{-1}\lambda_{\ell}=\lambda_s.
   \]
   This implies 
   \[
   \|\bv_s\|\le \max \{ \|\bw_s \|, \|[b_{is}b_{i{\ell}}^{-1}]\bw_{\ell}\| \}
   \le \lambda_s.
   \]
%   Since the successive minima $\lambda_1, \ldots, \lambda_n$  are uniquely determined by $\Lambda$, 
In view of Lemma \ref{lem;unique} the above inequality forces $\|\bv_s\|=\lambda_s$.
\end{proof}
 
\begin{prop}\label{lem;good basis}
Let  $\Lambda$  be a discrete and cocompact  $R$-module in $\K^n$. 
	There exists a basis of successive minima  $\bv_1, \ldots, \bv_n$ for $\Lambda$
 and a permutation $\sigma$ of coordinates of $\K^n$  such that the coordinates of 
 $\sigma(\bv_j)=(c_{1j}, \ldots, c_{nj})^{\mathrm{tr}}$ satisfy 
 \begin{align}\label{eq;cij}
|c_{ij}|<  |c_{ii}|=\lambda_i\qquad \mbox{ for all }  i\neq j. 
 \end{align}
\end{prop}
\begin{proof}
    We will show  by induction    that for all $1\le {\ell}\le n$ there  exist
    a basis of successive minima $\bv_1, \ldots, \bv_n$ 
    and  a permutation $\sigma$ of coordinates so that 
with  $\sigma(\bv_j)=(c_{1j}, \ldots, c_{nj})^{\mathrm{tr}}$
the followings hold:
    \begin{enumerate}[label=\textup{(\roman*)}] 
    \item  For each $1\le j\le {\ell}$ one has  
      $|c_{jj}|=\lambda_j $. 
	\item For each $1\le i\le {\ell}$ one has   $|c_{ij}|<  \lambda_i$ for all  $j\neq i$ and $j\le {\ell}$. 
    \end{enumerate}
The lemma follows from the case where ${\ell}=n$. 
 
   For ${\ell}=1$,  suppose $\bv_j=(b_{1j}, \ldots, b_{nj})^{\mathrm{tr}} \ (1\le j\le n) $ is a basis of successive minima 
    and $\|\bv_1\|= |b_{i1}|$. We take $\sigma$ to be the transposition of $1$ and 
    $i$ if $i\neq 1$ or identity otherwise. 
   Then it is clear that   (\rmnum{1}) and (\rmnum{2}) hold.
     
 Suppose for ${\ell}-1<n$ we can find a basis of  successive minima $\bw_1, \ldots, \bw_n$  and $\sigma$ so
    that (\rmnum{1}) and  (\rmnum{2})  hold
    for coordinates of $\sigma(\bw_j)=(b_{1j}, \ldots, b_{nj})^{\mathrm{tr}}$ with     $ {\ell}-1$   
    in place of ${\ell}$. 
   We show that we can go further to ${\ell}$. This induction process will complete the proof.

We claim that we can choose  $\bw_{\ell}$ so that $|b_{i{\ell}}|<\lambda_i$ for all $i<{\ell}$. To achieve this we shall replace  $\bw_{\ell}$ by 
$\bw_{\ell}- \sum_{j=1}^{{\ell}-1} a_j\bw_{j}$ where $a_j\in \F[t] $ will be determined later.  
There exists a unique  $(c_1, \ldots, c_{{\ell}-1})^{\mathrm{tr}}\in \K^{{\ell}-1}$ so that 
\[
\begin{pmatrix}
    b_{11} & \cdots & b_{1, {\ell}-1}\\
    \vdots & \ddots & \vdots \\
    b_{{\ell}-1,1} & \cdots & b_{{\ell}-1, {\ell}-1}
\end{pmatrix}
\begin{pmatrix}
    c_1 \\
    \vdots \\
    c_{{\ell}-1}
\end{pmatrix}
=\begin{pmatrix}
    b_{1,{\ell}}\\
    \vdots \\
    b_{{\ell}-1, {\ell}}
\end{pmatrix},
\]
since the determinant of the square matrix $B=(b_{i,j})_{i, j< {\ell}}$ has absolute value $\lambda_1\cdots \lambda_{{\ell}-1}>0$.  Let $a_i=[c_i]$ for $i<{\ell}$  and
write 
\begin{align}\label{eq;wk}
\sigma(\bw_{\ell})-\sum_{j=1}^{{\ell}-1} a_j \sigma(\bw_{j})=(b_{1{\ell}}', \ldots, b_{n{\ell}}')^{\mathrm{tr}}.
\end{align}
Then
\begin{align}\label{eq;claim}
|b_{i{\ell}}'|=|\sum_{s=1}^{{\ell}-1} b_{i s} (c_i-[c_i])|< \lambda_i,
\end{align}
since $|b_{is}|\le \lambda_i$ by induction hypothesis. 
Note that for $1\le j< {\ell}$ we have 
\[
|a_j|=|[c_j]|\le |c_j|\le \lambda_{\ell} \lambda^{-1}_{j}.
\]
It follows that \[
\|\bw_{\ell}- \sum_{j=1}^{{\ell}-1} a_j\bw_{j}\|\le \max \{  \|\bw_{\ell}\|, \|a_j\bw_j\|: 1\le j< {\ell} \}\le \lambda_{\ell}.
\]
In view of Lemma \ref{lem;unique} we have equality holds in the above inequality.
So  replacing 
$\bw_{\ell}$ by 
$\bw_{\ell}- \sum_{j=1}^{{\ell}-1} a_j\bw_{j}
$
does not change  the induction hypothesis and the claim holds by (\ref{eq;claim}).

 We  assume without loss of generality that  $\bw_{\ell}$ satisfies the claim above and
 choose $s\in \{\ell, \ldots, n \}$ so that $\| \bw_{\ell}\|= |b_{s{\ell}}|$.
 %We know from the claim 
%that $s\ge {\ell}$. 
Let $\tau$ be the transposition  of $s$ and ${\ell}$ if $s\neq {\ell}$ and identity otherwise. We take $\bv_j=\bw_{j}$ for $j\ge {\ell}$ and  
 \[
    \bv_j=\bw_j- [b_{sj}b_{s{\ell}}^{-1}] \bw_{\ell}.
j    \]
for $j< {\ell}$. 
In view of Lemma \ref{lem;basic} we know $\bv_1, \ldots, \bv_n$
is a basis of successive minima. 

By construction (\rmnum{1}) holds for 
 coordinates of $\tau \sigma(\bv_{\ell})$. When  $j<{\ell}$ and  $\lambda_j<\lambda_{\ell}$, we have 
  $\bv_j=\bw_j$. Hence (\rmnum{1}) holds for the $j$th 
coordinate of $\tau \sigma(\bv_j)$. When  $j<{\ell}$ and
    $\lambda_j=\lambda_{\ell}$, 
we have 
\begin{align}
    \label{eq;twice}
   |b_{sj} b_{s{\ell}}^{-1}|\le \lambda_j \lambda _{\ell}^{-1}=  1.  
\end{align}
 This together with  $|b_{j{\ell}}|<|b_{jj}|=\lambda_j$ of the  claim
  imply
 \[
|[b_{sj} b_{s{\ell}}^{-1}]b_{j{\ell}}|<\lambda_j.
\]
Hence
\[
|b_{jj}-[b_{sj}b_{s{\ell}}^{-1}]b_{j{\ell}}|=|b_{jj}|=\lambda_j,
\]
which is   (\rmnum{1}) for this $j$. 

For $i={\ell}$ and $j<{\ell}$, (\rmnum{2}) holds for coordinates of $\tau\sigma(\bv_j)$  by construction. For $i<{\ell}$ and $j={\ell}$  
the claim implies  (\rmnum{2}) holds.
For $i,j<{\ell}$ and  $i\neq j$ we have 
\[
|b_{ij}-[b_{sj}b_{s{\ell}}^{-1}]b_{i{\ell}}|\le \max\{ |b_{ij}|,|[b_{sj}b_{s{\ell}}^{-1}]b_{i{\ell}}|
\}<\lambda_i,
\]
where we use (\ref{eq;twice}), $  |b_{ij}|< \lambda_i$ and $|b_{i{\ell}}|<\lambda_i$.
This is (\rmnum{2}) for this $i$th component of $\tau\sigma(\bv_j)$.

\begin {comment}

  \red{I stopped here.}
    
   % For $k=1$ let $\bv_1=\bw_1$ and permutation $\sigma$ be the one that change the row, in which $\bw_1 $ attach the norm $\lambda_1$, to the first coordinate.
		
	Suppose for $1\leq k-1\leq n-1$, we have got $\bv_1,\ldots,\bv_{k-1}$ and permutation $\sigma$ meeting the property above. For $k$, let $\bv_{k}=\bw_{k}$ at the beginning, of course $\bv_1,\ldots,\bv_{k}\in R$-$[\bw_1,\ldots,\bw_{k}]$ linearly independent. For convenience we may ignore permutation $\sigma$ in the progress of induction and see $\bv_j=(c_{1j}, \ldots, c_{nj})^{\mathrm{tr}}$ for $1\leq j\leq k$ through permutation $\sigma$. Then for $\bv=(c_{1},\ldots,c_{n})^{\mathrm{tr}}$ with $\Vert \bv \Vert=\lambda_{k}$ let
    \[
        B(\bv)=\{j:1\leq j\leq k-1,|c_{j}|=\lambda_{k}\}.
    \] 
	$B(\bv_{k})$ may be empty, if so we jump to the next step. Otherwise let $ j=\min B(\bv_{k})$. Notice $|c_{j j}|=\lambda_j\neq 0$, $c_{j j}$ is invertible in $\K$. So assume that 
    \[
        c_{j k}=ac_{j j}+bc_{j j}
    \]
    with $a\in R$ and $|b|<1$. So 
    \[
        |bc_{j j}|<\lambda_j\leq \lambda_{k}. 
    \]
    And by $|c_{j k}|=\Vert \bv_{k} \Vert=\lambda_{k}$ we have:
    \[
        |ac_{j j}|=\lambda_{k} \text{ and } |a|=|c_{j k}|/|c_{j j}|=\lambda_{k}/\lambda_j.
    \]
    Let $\bv'_{k}=\bv_{k}-a\bv_{j}$. Suppose $\bv'_k=(c'_{1k}, \ldots, c'_{nk})^{\mathrm{tr}}$. Then 
    \begin{equation}
        \Vert \bv'_{k}\Vert\leq \max \{\Vert \bv_{k}\Vert,\Vert a\bv_{j}\Vert \}=\lambda_{k}.
    \end{equation}
    Trivially $\bv_1,\ldots,\bv_{k-1},\bv'_{k}\in R[\bw_1,\ldots,\bw_{k}]$ linearly independent. Thus by the definition of $\lambda_{k}$ and (3.1), we have 
    \[
        \Vert \bv'_{k}\Vert=\lambda_{k}.
    \]  	
    Thus we can repeat the step above again replacing $\bv_{k}$ with $\bv'_{k}$ to get a new $\bv$ which satisfies $\bv_1,\ldots,\bv_{k-1},\bv\in R[\bw_1,\ldots,\bw_{k}]$ linearly independent and $\Vert \bv\Vert=\lambda_{k}$. And I will show that $\min B(\bv'_{k})> \min B(\bv_{k})$(if $B=\emptyset$, set $\min B =\infty$). 
		
	For $j$,
    \[
        |c'_{j k}|=|bc_{j j}|<\lambda_j\leq \lambda_{k}
    \]
    So $j\notin B(\bv'_{k})$.
		
	For $s<j$, we know $\lambda_s<\lambda_j$. By $s\notin B(\bv_{k})$ and $|c_{s j}|<\lambda_{s}$, we have
    \[
        |c'_{s k}|=|c_{s k}-ac_{s j}|\leq \max\{|c_{s k}|,|ac_{s j}|\}<\max\{\lambda_{k},\frac{\lambda_{k}}{\lambda_j}\lambda_{s}\}=\lambda_{k}
    \]
    So $s\notin B(\bv'_{k})$.
	
    As a consequence we know that $\min B(\bv'_{k})> \min B(\bv_{k})$. Notice $B(\bv_{k})$ is contained in a finite set, so by repeating the step above in finite times we can get a $\bv_{k}$ satisfies $\bv_1,\ldots,\bv_{k-1},\bv_{k}\in R[\bw_1,\ldots,\bw_{k}]$ linearly independent with $\Vert \bv_{k}\Vert=\lambda_{k}$ and $B(\bv_{k})=\emptyset$. Thus there exits $k-1<j\leq n$ such that $c_{jk}=\lambda_{k}$. Through a permutation we may change $j$ to $k$, and for convenience, we just assume $j=k$ with the permutation just in mind. Here we furthermore have $|c_{j k}|<\lambda_{k}$ for $1\leq j< k$, and denote this property by ($\star$).
		
	Then I will change $\bv_{k}$ so that $\bv_{k}$ can satisfy (2). Define $C(\bv_{k})=(m,s)$ where $m,s$ are defined as follow:
	\[
		m=\max\{|c_{j k}|:1\leq j\leq k-1, |c_{j k}|\geq \lambda_j\}\cup\{0\}  
    \]
    and 
    \[
        s=\min\{j:1\leq j\leq k-1,|c_{j k}|=m, |c_{j k}|\geq \lambda_j\}\cup\{n+1\}.
    \]
     If $C(\bv_{k})=(0,n+1)$, jump to the next step. Otherwise like before suppose 
     \[
        c_{s k}=ac_{s s}+bc_{s s}
     \]
    where $a\in R$ and $|b|<1$. By ($\star$), $\lambda_{k}>|c_{s k}|=m\geq \lambda_s$. And we have 
    \[
        a\neq 0 \text{ and }|a|=\frac{m}{\lambda_{s}}<\frac{\lambda_{k}}{\lambda_{s}}.
    \]
    Let $\bv'_{k}=\bv_{k}-a\bv_{s}$, and suppose $\bv'_k=(c'_{1k}, \ldots, c'_{nk})^{\mathrm{tr}}$ then
    \[
        \Vert \bv'_{k}\Vert=\Vert \bv_{k}-a\bv_{s}\Vert\leq \max\{\Vert \bv_{k}\Vert,|a|\Vert \bv_{s}\Vert\}\leq \lambda_{k}.
    \]
    And by $|a|<\frac{\lambda_{k}}{\lambda_{s}}$, 
    \[
        |ac_{ks}|<\frac{\lambda_{k}}{\lambda_{s}} \lambda_{s}=\lambda_{k}=|c_{kk}|.
    \]
    So
    \[
        |c'_{kk}|=|c_{kk}-ac_{ks}|=|c_{kk}|=\lambda_{k}.
    \]
    So (1) holds. And for $1\leq j\leq k-1$ 
    \[
        |c'_{j k}|=|c_{jk}-ac_{js}|\leq \max\{ |c_{jk}|,|ac_{js}|\}<\lambda_{k}.
    \]
    So ($\star$) holds. Then I'm going to calculate $C(\bv'_{k})=(m',s')$. We may divide the first $k$ rows into 4 parts.

    For j satisfies $|c_{j k}|>m$, by definition of m, we have $m<|c_{j k}|< \lambda_j$ and $j\neq s$. Thus 
    \[
    |c'_{j k}|=|c_{j k}-ac_{j s}|\leq\max\{|c_{j k}|,\frac{m}{\lambda_s}\lambda_{s}\}<\lambda_j,
    \]
    which means j isn't in the set we calculate $m'$.

    For j=s, we have
    \[
    |c'_{s k}|=|c_{s k}-ac_{s s}|=|bc_{s s}|<\lambda_{s},
    \]
    similarly s isn't in the set we calculate $m'$.

    And for j isn't contained in the above 2 parts, we know either $j<s$ and $|c_{j k}|<m$ or $j>s$ and $|c_{j k}|\leq m$. 
    
    For $j<s$ and $|c_{j k}|<m$, by $\lambda_{j}\leq \lambda_{s}$ and $|c_{j s}|<\lambda_{j}$ we have 
    \[
    |c'_{j k}|=|c_{j k}-ac_{j s}|\leq \max\{|c_{j k}|,\frac{m}{\lambda_s}(\lambda_{j}-1)\}<\max\{m,\frac{m}{\lambda_s}\lambda_{s}\}=m.
    \]
    
    For $j>s$ and $|c_{j k}|\leq m$, we have
    \[
     |c'_{j k}|=|c_{j k}-ac_{j s}|\leq \max\{|c_{j k}|,\frac{m}{\lambda_s}\lambda_{s}\}\leq m.
    \]
	As a result, corresponding to the last two parts we have either $m'<m$ or $m'=m$ and $s'>s$, thus by finitely repeating the step above we can get $\bv_{k}$ satisfies $|c_{j k}|<\lambda_j$, for any $1\leq j\leq k$.
		
	Finally we change $\bv_1,\cdots,\bv_k$ to get the left part of (2).
		
	If $|c_{k j}|\geq \lambda_{k}$ for $1\leq j\leq k$, by $|c_{k j}|\leq\lambda_j\leq \lambda_{k}$ we have $|c_{k j}|=\lambda_{j}= \lambda_{k}$. Suppose 
    \[
        c_{k j}=ac_{k k}+bc_{k k}
    \]
    where $a\in R$ and $|b|<1$. Like before we have $|a|=1$. Let $\bv'_{j}=\bv_{j}-a\bv_{k}$. So
    \[
        \Vert \bv'_{j}\Vert\leq \max\{\Vert \bv_{j}\Vert ,\Vert \bv_{k}\Vert \}=\lambda_{j}.
    \]
    Like before this means (1) still holds. And
	\begin{equation}
		|c'_{s j}|=|c_{s j}-ac_{s k}|\left\{ \begin{array}{ll}
			<\lambda_{s} &s\neq j<k\\
			=\lambda_j &s=j \\
			= |bc_{k k}|<\lambda_{k} &s=k
		\end{array}\right.\nonumber
	\end{equation}
    In the first two condition we use $|c_{s k}|<\lambda_{s}$. Thus (1) still hold, and by finitely steps we can get (2). Which means we can get $\bv_1,\cdots,\bv_{k}$ satisfying (1)(2). 
    
    Finally by induction and let $\{c_{j}\}_{j=1}^{n}$ and $\sigma$ be the corresponding vectors and permutation that (1)(2) hold when k=n, we complete the proof.

    \end{comment}
\end{proof}

Let $R^n$ be the standard cocompact free $R$-module of $\K^n$ consists of vectors with entries in $R$.  
Given $g\in \GL_{n} (\K)$, let $\Lambda_g=g{R}^{n}=\{g\bv: \bv\in R^n \}$.  Since $\det (g)\neq 0$, we know $\Lambda_g$ is a discrete and cocompact $R$-modules in $\K^n$. Conversely, any discrete and cocompact $R$-module of $\K^n$ has this form. 

\begin{cor}\label{cor;det}
    Suppose $g\in \GL_{n} (\K)$ and $\lambda_{i}=\lambda_{i}(\Lambda_g)$. Then ${\lambda}_{1}{\lambda}_{2}\cdots{\lambda}_{n}=\lvert\det (g)\rvert$.
\end{cor}
\begin{proof}
    According to Proposition \ref{lem;good basis} there exists a permutation 
    matrix $\sigma \in \GL_{n} (\K)$ and $\gamma\in \GL_{n}(R)$
    such that 
    \begin{align}\label{eq;cij1}
        \sigma g\gamma=(c_{ij}),
    \end{align}
    where $c_{ij}$ satisfies (\ref{eq;cij}). 
    Since $|\det(\sigma)|=|\det (\gamma)|=1$, we have 
    \[
    |\det (g)|=|\det(c_{ij})|=|c_{11}\cdots c_{nn}|=\lambda_1\cdots\lambda_n,
    \]
    where we expand the determinant and use the ultrametric property of the norm.
\end{proof}

\begin{proof}[Proof of Theorem \ref{thm:AKT}]
     Let $ g\in G$ and $\lambda_{i}=\lambda_{i}(\Lambda_g)$. Let  $\sigma, \gamma  $ and $(c_{ij})$ be as in the proof of Corollary \ref{cor;det}
    so that     (\ref{eq;cij1}) holds.  By possiblly replacing  $\gamma$ by 
    $\gamma \diag(\det(\gamma \sigma )^{-1} , 1, \ldots, 1),$ we assume further that 
    $\gamma \sigma \in \Gamma$ and hence $\sigma  g \gamma \in G$. 
    It follows that 
    \[
        b:= c_{11}\cdots c_{nn}\equiv 1\mod \mathfrak p. 
    \]
Let $a=\diag(c_{11}^{-1}b, c_{22}^{-1},  \ldots, c_{nn}^{-1})\in A$. We have 
    \[
    a  \sigma g \gamma  =h\in  \SL_n(\mathfrak o;\mathfrak p).
    \]
    Hence
    \[
    \sigma a \sigma \cdot g \cdot \gamma \sigma = \sigma h \sigma 
    \in \SL_n(\mathfrak o;\mathfrak p), 
    \]
    which gives the required decomposition of the theorem. 
\end{proof} 
\begin{cor}
    \label{cor;use}
    $\mathfrak m_n=\inf \{\bN (L_h)^{-1}: h\in  \SL_n(\mathfrak o;\mathfrak p) \}$. Therefore,  to estimate $\mathfrak m_n$, it suffices to estimate those $\bN (L_h)$ with $h\in \SL_n(\mathfrak o;\mathfrak p) $.
\end{cor}
\begin{proof}
Let $L=L_g$ be a product of linear forms given by $g\in G$. By Theorem \ref{thm:AKT} we can write 
$g= a h \gamma$ where $a\in A, \gamma \in \Gamma$ and $h\in 
\SL_n(\mathfrak o;\mathfrak p)$. Note that $\bN(L_g)=\bN(L_h)$, so it suffices to estimate the latter. 
\end{proof}
Using Theorem \ref{thm:AKT} and its corollary, we get our first   lower bound of $\mathfrak m_n$.
\begin{prop}\label{prop;first}
    $\mathfrak{m}_n\geq q^{ n-1}$.
\end{prop}
\begin{proof}
Let $h\in \SL_n(\mathfrak o;\mathfrak p) $.
 For  $\be_1=(1,0,  \ldots, 0)^{\mathrm{tr}}\in R^n$ we have 
\[
|L_h(\be_1)|\le q^{-(n-1)}
\]
by the definition of $\SL_n(\mathfrak o;\mathfrak p)$. 
%It follows that $\bN (L_h)\le q^{-(n-1)}$ for all $g\in G$. 
The conclusion follows from this and 
Corollary \ref{cor;use}.
\end{proof}
This lower bound coincides with the upper bound when $n \le q+1$. So we try to improve it when $n>q+1$. 

\begin{prop}\label{prop;second}
We have $\log _q\mathfrak m_n \ge 
%(n-2)+ {\left\lceil  \frac{n-2}{q-1}  \right\rceil }
 n-2+\left\lceil\frac{n}{q+1} \right\rceil
$, where $\lceil  
* \rceil $ is the smallest integer greater than or equal to
$\lceil *
\rceil$.
\end{prop}
\begin{proof}
    Let $g=(b_{ij})\in \SL_n(\mathfrak o;\mathfrak p)$. 
    We assume without loss of generality that $b_{ij}\neq    0$, since otherwise
    $\bN (L_g)=0$. We also assume $n>q+1$ since otherwise Proposition \ref{prop;first} gives the conclusion.
    For each $1\le i\le n$, we let $\theta_i\in \F\cup \{\infty\}$ so that 
    $b_{i2} b_{i1}^{-1}\equiv \theta  \mod \mathfrak p$ where $\theta_i=\infty$
    if $|b_{i2}|>|b_{i1}|$. 
     There are at least 
    \[
    r=\left\lceil\frac{n}{q+1} \right\rceil\ge 2
    \]
    $\theta_i$ with the same value $\theta$. 
    
      If $\theta =\infty$, then in view of the definition of $\SL_n(\mathfrak o;\mathfrak p)$ we have there are at least 
      $r-1$ entries of $(b_{11}, \ldots, b_{n1})^{\mathrm{tr}}$  with 
      $|b_{i1}|\le q^{-2}$. So we have 
      \[
      \bN(L_g)\le |L_g(\be_1)|\le q^{-(n-1) -(r-1)}=q^{-(n-2+r)}.
      \]
    The case $\theta=0$ can be estimated similarly and  
    \[
    \bN(L_g)\le |L_g(\be_2)|\le q^{-(n-1)- (r-1)}=q^{-(n-2+r)}.
    \]
    If $\theta\neq 0, \infty$, then 
    \[
        \bN(L_g)\le |L_g(\be_2-\theta \be_1)|\le q^{-(n-2 +r)}.
    \]
    Therefore, we have $\bN(L_g) \le q^{-(n-2 +r)}$. Since $g$ is arbitrary we have
$\log _q\mathfrak m_n \ge n-2+r$.

\end{proof}

For $\bv=(v_1,\dots,v_m)^{\mathrm{tr}}\in \K^{m}$ let $N(\bv)=|v_1\cdots v_m|$.
\begin{lem}\label{lem;trend}
    Let $\bv_1,\dots ,\bv_{\ell}\in \mathfrak {o}^{m}$, then there exists $(a_1,\dots,a_\ell)\in \F^{k}\backslash\{0\}$ such that 
    \[
    -\log_{q}N(a_1\bv_1+\cdots+a_\ell\bv_\ell)\ge 
    \left (\frac{1}{q-1}-\frac{\ell}{q^{\ell}-1}\right )m.
    \]
\end{lem}
\begin{proof}
    Suppose $\bv_{i}=(v_{1i},\dots,v_{mi})^{\mathrm{tr}}$. For $1\leq j\leq m$ and $a=(a_1,\dots,a_\ell)\in \F^{\ell}\backslash{0}$, let $s(j,a)=-\log_{q}|a_{1} v_{j1}+\cdots+a_{\ell}v_{j{\ell}}|\ge 0$. Then  we have
    \[
    -\log_{q}N(a_1\bv_1+\cdots+a_\ell\bv_\ell)=\sum_{j=1}^{m}s(j,a).
    \]

    Fix $1\leq j\leq m$ and $1\leq r\leq \ell$,
    if we want $s(j, a)\ge r$, then we need to solve at most $r$ linear equations. 
    Therefore, 
    \[
    \# \{a\in \F^{\ell}\backslash\{0\}:s(j,a)\geq r\}\geq q^{\ell-r}-1.
    \]
    Hence
    \begin{eqnarray}
        \sum_{a\in \F^{\ell}\backslash\{0\}}s(j,a)&=&\sum_{r=1}^{\infty}r\# \{a\in \F^{\ell}\backslash\{0\}:s(j,a)=r\}\nonumber \\ 
        &=&\sum_{r=1}^{\infty}\# \{a\in \F^{\ell}\backslash\{0\}:s(j,a)\geq r\}\nonumber\\ 
        &\geq&\sum_{l=1}^{\ell}q^{\ell-r}-1\nonumber\\ 
        &=& \frac{q^{\ell}-1}{q-1}-\ell\nonumber.
    \end{eqnarray}
    So
    \[ 
    \sum_{a\in \F^{\ell}\backslash\{0\}}\sum_{j=1}^{m}s(j,a)=\sum_{j=1}^{m}  \sum_{a\in \F^{\ell}\backslash\{0\}}s(j,a)\geq \left (\frac{q^{\ell}-1}{q-1}-\ell \right)m.
    \]
    Thus there exits an $a=(a_1,\dots,a_\ell)=\in \F^{\ell}\backslash\{0\}$ such that
    \[
    -\log_{q}N(a_1\bv_1+\cdots+a_\ell\bv_\ell)=\sum_{j=1}^{m}s(j,a)\ge  \frac{ (\frac{q^{\ell}-1}{q-1}-\ell)m}{q^{\ell}-1}=
    \left (\frac{1}{q-1}-\frac{\ell}{q^{\ell}-1}\right )m.
    \]
\end{proof}
\begin{prop}\label{prop;trend with r}
For any integer  $1<\ell<n$, we have
    \[
    \log_{q}\mathfrak{m}_n\geq (n-\ell)+
    \left (\frac{1}{q-1}-\frac{\ell}{q^{\ell}-1}\right )(n-\ell).
    \]
    
\end{prop}
\begin{proof}
    Let  $g=(g_{ij})\in \SL_n(\mathfrak o;\mathfrak p)$ and let $\bw_j \ (1\le j\le n) $ be the column vectors of $g$.
    Applying  Lemma \ref{lem;trend} to the vectors 
    $$\bv_j= (t g_{\ell+1,j}, \ldots, t g_{n,j})^{\textup{tr}}\in \mathfrak o^{n-\ell }\qquad (1\le j\le \ell), $$
    %    consider the first $\ell$ columns of $g$, suppose they are $\bv_1,\dots ,\bv_{\ell}$. Notice the $1^{\mathrm{st}},\dots,{\ell}^{\mathrm{th}}$ coordinates of them are in $\mathfrak{o}$ and ${\ell+1}^{\mathrm{th}},\dots,n^{\mathrm{th}}$ coordinates of them are in $\mathfrak{p}$. So applythe last $n-\ell$ coordinates, 
    we know that there exists $(a_1,\dots,a_\ell)\in \F^{\ell}\backslash\{0\}$ such that
    \[
    -\log_{q}N(a_1\bw_1+\cdots+a_\ell\bw_\ell)\ge 
     (n-\ell)+(\frac{1}{q-1}-\frac{\ell}{q^{\ell}-1})(n-\ell).
    \]
  It follows that 
    \[
    -\log_{q}\bN(L_{g})\ge  (n-\ell)+
    \left (\frac{1}{q-1}-\frac{\ell}{q^{\ell}-1}\right )(n-\ell).
    \]
 The conclusion follows from this and Corollary \ref{cor;use}.
\end{proof}
\begin{proof}
    [Proof of Theorem \ref{thm;main}]
    When $n\le q$
    the lower bound of $\log_q \mathfrak m_n $ given in the theorem follows from 
    Proposition \ref{prop;second}. Assume now $n>q$. Using Proposition \ref{prop;trend with r} for 
      ${\ell}=\lceil \log _q n \rceil\ge 2 $, we get 
     \begin{align*}
     \log _q \mathfrak m_n &\ge (n-\lceil \log _q n \rceil)
   \left  ( \frac{q}{q-1}- \frac{\lceil \log _q n \rceil}{n-1}\right)\\
    &\ge \frac{nq}{q-1}-\lceil \log _q n \rceil
    \left (\frac{n-\lceil \log _q n \rceil}{n-1}+\frac{q}{q-1}\right)\\
    &\ge \frac{nq}{q-1}-\frac{2q}{q-1} \lceil \log _q n \rceil\\
    & \ge \frac{nq}{q-1}-\frac{3q}{q-1}  \log _q n .
     \end{align*}
It follows that 
\[
    \liminf_{n\rightarrow\infty} \frac{\log_{q}\mathfrak{m}_n}{n}\geq \frac{q}{q-1}.
    \]
\end{proof}

\section{Characterization of periodic orbits}\label{section;periodic}

The aim of this section is to prove Theorem \ref{thm;equivalent}. 
We first show that (\rmnum{1}) implies (\rmnum{2}).

Let $\alpha_1, \ldots, \alpha_n $ be an $\F(t)$-basis of a degree $n$ extension $k/\F(t)$ split at $\infty$.  Let $\sigma_1, \ldots , \sigma_n$ be distinct  embeddings of  $k$ into $\K$. 
Suppose $g\in G$ and 
$L_g$  is proportional to (\ref{eq;rational}). 
Since $\K[x_1, \ldots, x_n]$ is a unique factorization domain, there exists 
$u\in N_{GL_{n}(\K)}(A)=\{g\in GL_{n}(\K):gAg^{-1}=A\}$ such that 
   \begin{align}\label{eq;temp}
    ug=  \begin{pmatrix}
    \sigma_1(\alpha_1)&  \cdots &\sigma_1(\alpha_n)\\
    \vdots & \ddots & \vdots \\
    \sigma_n(\alpha_1) & \cdots & \sigma_n(\alpha_n)
    \end{pmatrix}.
     \end{align}
     
Let $\mathfrak o_k=\{\alpha\in k: \alpha \mbox{ is integral over }\F[t] \}$. 
It follows  from the function field analogue of Dirichlet's unit 
theorem (\cite[Proposition 14.2]{rosen}) that 
\begin{align}\label{eq;group}
 \{\diag(\sigma_1(\alpha), \ldots, \sigma_n(\alpha)): \alpha \in 
\mathfrak o_k, N_{k/\F(t)}(\alpha)=1 \}    
\end{align}
is a lattice of $A$. The group in (\ref{eq;group}) is contained in the
$\textup{Stab}_A(g\Gamma)$, so $Ag\Gamma$ is periodic in  $G/\Gamma$. 

Now we assume $Ag\Gamma$ is periodic in $G/\Gamma$ and prove (\rmnum{2}) 
implies (\rmnum{1}).
By definition 
 $\Stab_{A}(g\Gamma)$ is a lattice of  $A$. We fix a uniformizer $\varpi$ of $\K$.    
 Let $q: \K^*\to \Z$ be the homomorphism defined by $q(b)=\log_q |b|$.
 %For every $\alpha\in \K^* $ there is a uniquely determined integer $q(\alpha)$ such that 
  %$\alpha\varpi ^{-q (\alpha)} \in \mathfrak o^*$. 
  The map 
 \[
 \pi: A\to H:=\{ (m_1, \ldots,m_n)\in \Z^n: \sum_{i=1}^n m_i=0 \} 
 \]
 defined by $\pi(\diag(a_1, a_2, \cdots, a_n))= (q(a_1), \ldots, q(a_n))$
 is a proper surjective group homomorphism. Hence $\pi (\Stab_{A}(g\Gamma))$ has finite 
 index in $H$. So there exists $(m_1, \ldots,m_n)\in \pi (\Stab_{A}(g\Gamma))$
 such that $\sum_{i\in I}m_i\neq \sum_{j\in J}m_j$ for all nonempty  disjoint sets $I, J\subset
 \{1, 2,\ldots, n\}$. 
It follows that there exists 
    \begin{align}\label{eq;aij}
        a=\diag(a_1, a_2, \cdots, a_n)\in \Stab_{A}(g\Gamma) \mbox{ with }
        \prod_{i\in I}|a_i|\neq \prod_{j\in J} |a_j|
        \mbox{ for } I\cap J=\emptyset.
    \end{align}
    
Note that  $ \Stab_{A}(g\Gamma)= g\Gamma g^{-1}\cap A$. So  there exists $\gamma\in \Gamma$ such that 
\begin{align}\label{eq;ggamma}
a=g\gamma g^{-1}.
\end{align}
So the characteristic polynomial  $f(x)=\det(xI-a)=\det (xI -\gamma)$ has coefficients in $R$. We claim that $f$ is irreducible in the 
ring $\F(t)[x]$. Otherwise, $f=f_1f_2$ where $f_1, f_2\in R[x]$ are monic polynomials with degree $>1$. Since the constant term of $f$ is $1$, the constant terms of $f_1$ and $f_2$ are in $\F^*$. So the product of roots of $f_i$  has absolute value $1$, which contradicts (\ref{eq;aij}). This proves the claim.  

 Let $k=\F(t, a_1)$ where $a_1$ is the first diagonal entry of $a$ in (\ref{eq;aij}). 
The claim above implies that  $k/\F(t)$ is a degree $n$ extension  split at $\infty$. 
For $1\le i\le n$, let  $\sigma_{i}$ be the distinct embeddings of $k$ into $\K$ defined by  $\sigma_{i}(a_1)=a_i$. Let $ (\alpha_1, \ldots, \alpha_n)\in k^n$ be a left  eigenvector of  $\gamma$ with eigenvalue $a_1$, i.e. 
    \[
        a_1 (\alpha_1, \ldots, \alpha_n) =(\alpha_1, \ldots, \alpha_n)\gamma.
    \]
Let $h$ be defined as the right hand side of (\ref{eq;temp}). Then we have
\[
a h = h\gamma.
\]
On the other hand in view of  (\ref{eq;ggamma}), there exist nonzero numbers 
 $c_1, \ldots, c_n\in  \K$ such that
    \[
   g=  \diag(c_1, \ldots, c_n)  h.
    \]
Therefore, $L_{g}= c_1\cdots c_n L_h$, which completes the proof of (\rmnum{2})
implies (\rmnum{1}).

\section{Algebraic Theory of function fields}
In this section, we review algebraic number theory and arithmetic geometry of a finite geometric extension $k/\F(t)$. Here geometric means $\F$ has no finite extensions in $k$ other than itself. 
More details can be found in \cite{lang1994algebraic}, \cite{rosen} and \cite{salvador2006topics}. This allows us to get a relation between $|d_k|$ and   the  genus of $k$ in Proposition \ref{prop;main}.
 Then we use results in arithmetic geometry to determine when the minimal genus 
 is equal to zero or one.

Recall that a valuation (or place) of $k$ is a surjective  homomorphism $v: k^* \to \Z$ such that $v(x+y)\ge \inf \{  v(x), v(y) \}$.
Here $k^*$ is the multiplicative group of nonzero elements of $k$. 
We extend $v$ to all of $k$ by  setting $v(0)=\infty$. 
A valuation $v$ of $k$ gives an absolute value $|\cdot|_v$ of $k$ with $|x|_v= q^{-v(x)}$ where $q=p^e=\sharp\, \F$.  
Let $k_v$  be the completion  of $k$  with respect to $|\cdot|_v$.
The absolute value  $|\cdot|_v$ and hence the valuation  $v$ extends naturally to $k_v$.
The residue field $F(v)$ of $v$ is the quotient of the valuation ring 
$\mathfrak o_v:=\{x\in k_v: v(x)\ge 0 \}$ by its maximal ideal. 

Let $\mathcal{P}_k$ be the set of  valuations of $k$. The free abelian group generated by $\mathcal{P}_k$ is called the divisor group of $k$ and is  denoted by $\textup{Div}(k)$. 
 Every divisor $D$ can be uniquely written as
\[D=\sum_{v\in\mathcal{P}_{k}}v(D)D_v,\]
where  $D_v \ (v\in \mathcal P_k)$ is the free basis. 
Let $\textup{deg\,}{D_v}=[F(v):\F]$ be the degree of $v$ and extend it linearly to a map $\textup{deg}: \textup{Div}(k)\to \Z$. 

\begin{prop}\cite[Proposition 5.1]{rosen}\label{prop,divisor}
Let $(x)_{k}=\sum_{v\in \mathcal{P}_{k}}v(x)D_v$ be the principal divisor given by $x\in k^*$. 
    Write it as  $(x)_k= (x)_{k,o}-(x)_{k, \infty}$ where 
    \[
    (x)_{k,o}=\sum_{v\in \mathcal{P}_k, v(x)> 0}v(x)D_v\quad\mbox{and}\quad  
    (x)_{k, \infty}=-\sum_{v\in \mathcal{P}_k, v(x)< 0}v(x)D_v.
    \]
If $x\in k\setminus \F$, then we have $[k:\F(x)]=\textup{deg\,}
(x)_{k,o}=\textup{deg\,}(x)_{k,\infty}$. In particular, we have $\textup{deg\,}(x)_{k}=0$ for all $x\in k^*$.
\end{prop}

A divisor $D$ is said to be effective if $v(D)\ge 0$ for all $v\in \mathcal P_k$. 
In this case we  write  $D\ge 0$ .
Let   $L(D)=\{x\in k: (x)_k+D\ge 0\}$. It follows from the definition of the valuation that $L(D)$ is a vector space over $\F$. 
Let  $l(D)=\dim_{\F}L(D)$.
The genus   $g_k$ of the field $k$ is 
 $$
  1+\max\{\textup{deg}(D)-l(D):D\in \textup{Div}(k)\}.$$ 
It 
 is always nonnegative since $L(0)=\F$. The genus of $\F(t)$ is zero. 
Later we will use the following corollary of the Riemann-Roch theorem.
\begin{lem}\label{lem;rr}
Let $D\in \textup{Div}(k)$ and $\textup{deg}(D)>2g_k-2$. Then $l(D)=\textup{deg}(D)-g_k+1$.  
\end{lem}

Let  $v$ and  $w$  be  valuations  of $k$ and $\F(t)$, respectively.
We write  $v/w$ if there exists a positive integer $e_{v/w}$ such that 
$v|_{\F(t)}=e_{v/w} w$. We say $v$ is over $w$ or $w$ is below $v$. 
The number   $e_{v/w}$  is the  ramification index of  
$v/w$. In this case, $F(w)$ can be naturally identified with a sub-field of $F(v)$ and   $f_{v/w}=[F(v):F(w)]$ is  the inertia degree of $v/w$. 
Recall that there is a group homomorphism $N_{k/\F(t)}: \textup{Div}(k)\to \textup{Div}{(\F(t))}$ given by $ N_{k/\F(t)}(D_v)= {f_{v/w}} D_w$ for every prime divisor $v/w$. 
Let  $w_{\infty}$ be  the valuation of $\F(t)$
given by $w_\infty(t^{-1})=1$.
Let $\mathfrak{D}_{k/\F(t)}$
%$=\sum_{v\in \mathcal{P}_k}v(\mathfrak{D}_{k/\F(t)})D_v$
be the different divisor of $k/\F(t)$ defined by 
 \[v(\mathfrak{D}_{k/\F(t)})= \max\{-v(x): \textup{tr}_{k_v/\F(t)_w}(x )\in \mathfrak o_w  \}\] 
 where $w\in \mathcal P_k$ is below $v$. 
We need the following special case of the Riemann-Hurwitz genus formula.
 \begin{prop}\label{thm;R-H thm}\cite[Theorem 7.16]{rosen}
Let $k/\F(t)$ be a finite geometric separable extension of degree n. Let $\mathfrak{D}_{k/\F(t)}$ be the  different divisor of $k/\F(t)$, then
\[2g_{k}-2=-2n+\textup{deg}(\mathfrak{D}_{k/\F(t)}).\]
\end{prop}

From now on we assume  the extension $k/\F(t)$ is split at $\infty$.  
Recall that this is equivalent to 
 $e_{v/w_{\infty}}=1$ and $f_{v/w_{\infty}}=1$  for every valuation $v\in\mathcal P_k$ with $v/w_\infty$. This implies that $k/\F(t)$ is geometric and separable.

 The discriminant divisor of $k/\F(t)$ is $\mathfrak d_{k/\F(t)}= N_{k/\F(t)}(\mathfrak{D}_{k/\F(t)})$. Since $k/\F(t)$ is geometric, we have 
\begin{align}\label{eq;Dk}
\deg _k \mathfrak{D}_{k/\F(t)}=\deg_{\F(t)} \mathfrak d_{k/\F(t)}
\end{align}
according to \cite[Proposition 7.7]{rosen}. Recall that $d_k$ is the discriminant of $k$ defined in the introduction. Since $k/\F(t)$ splits at infinity we have  
\begin{align}\label{eq;dk}
\mathfrak d_{k/\F(t)}=(d_k)_{\F(t),o}.
\end{align}
For a proof of this see \cite[p. 201]{Neukirch}.
 %and \cite[p.155]{salvador2006topics}. 
 In view of (\ref{eq;Dk})
 and (\ref{eq;dk}) we have 
 \begin{align}\label{eq;finish}
 \deg _k \mathfrak{D}_{k/\F(t)}=\deg (d_k)_{\F(t),o}= -w_\infty (d_k), 
 \end{align}
 where the last equality follows from Proposition \ref{prop,divisor}.
Using (\ref{eq;finish}), Proposition \ref{thm;R-H thm} and the fact $|d_k|=q^{-w_\infty (d_k)}$, we get the following formula of $|d_k|$.
\begin{prop}\label{prop;main}
Suppose $k/\F(t)$  is a degree $n$ extension split at $\infty$. Then $|d_k|= q^{2n+2 g_k-2} $.
\end{prop}

 Proposition \ref{prop;main} reduces to the estimate of the minimal  discriminant to  calculating the minimal genus of given $n$. Let 
\[g_{\min}(n,q)=\min \{g_k: [k:\F(t)]=n,  k \text{ is split at }\infty\},\]
then 
\begin{align}\label{eq;use}
    \textup{disc-}\mathfrak m_n=\min |d_k|^{1/2}=q^{n-1+g_{\min}(n,\F)}
\end{align}
by Proposition \ref{prop;main}.
 In the remaining of this section we  use results in  arithmetic geometry to
 calculate  the minimal genus. 
%As an application of   the Riemann Hypothesis on the function field

\begin{lem}[\textbf {Serre Bound}]\label{thm; H-W Bound}
Let  $N_{k}(\mathbb{F})$ be the number of the prime divisors of degree one on $k$. Then we have
\[ |N_{k}(\mathbb{F})-q-1|_{\R}\leq \lfloor 2\sqrt{q}\rfloor g_{k}.\]
Here $|~~|_\R$ is the usual absolute value in $\R$ and $\lfloor x \rfloor$ is the largest integer $\leq x$.
\end{lem}

\begin{lem}\label{thm;ineq genus}
    $g_{\min}(n,q)\ge (n-q-1)/\lfloor 2\sqrt{q}\rfloor $ and $g_{\min}(n,q)=0$ if and only if $2\leq n\leq q+1$. 
\end{lem}
\begin{proof}
    Suppose  $k$ is split at infinity, then  $ N_{k}(\F)\ge n$.
    Therefore, 
    \[
    |N_{k}(\F)-q-1|_\R \ge N_{k}(\F)-q-1\ge n-q-1.
    \]
    This together with Theorem 2.14 implies 
    \[
   \lfloor 2\sqrt{q}\rfloor g_k\ge  n-q-1.\]
   By considering $k$ with $g_k= g_{min}(n,q)$ we get 
   the first conclusion. In particular,  this implies 
    $g_{min}(n,q)>0$ when $n> q+1$. 
To complete the proof we only need to show $g_{min}(n, q)=0$ when  $2\le n\leq q+1$.
It suffices to find an embedding of $\F(t)$ into the rational function field $k=\F(y)$ so that $k$ is a degree $n$ extension split at infinity. 
%   This will complete the proof. 

   Let $D_{v_\infty}$ and $D_{v_i}\ (1
    \leq i\leq q)$ be the prime divisors in $k$ of degree one, such that $v_{\infty}(y)=-1$ and $v_i(y-c_i)=1\ (1
    \leq i\leq q)$ where $\{c_i\}_{i=1}^q$ are distinct elements in $\F$.
     For $2\leq n\leq q$, let 
    $t=\prod_{i=1}^{n}(y-c_i)^{-1}$, then
    \[(t)_{k,\infty}=\sum_{i=1}^nD_{v_i}.\]
    By Proposition \ref{prop,divisor} we have $[k:\F(t)]=n$.
     The minimal polynomial of $y$ is \[
    f(Y)=(Y-c_1)\cdots (Y-c_n)-t^{-1}
    \]
    which has $n$ distinct roots in $\F$. Therefore, $f$ splits over $\F((t^{-1}))$ by Hensel's lemma \cite[p. 129]{Neukirch}. This implies 
    that $k$ splits at infinity and hence $g_{\min}(n,q)=0$.

    For $n=q+1$, let $t=(y^{q+1}-y^2+1)/\prod_{i=1}^{q}(y-c_i)$, then
    \[(t)_{k,\infty}=\sum_{i=1}^qD_{v_i}+D_{v_{\infty}}.\]
     By Proposition \ref{prop,divisor} we have  $[k:\F(t)]=q+1$. The element
    $y$ is the zero of the polynomial \[
    f(Y)= t^{-1}(Y^{q+1}-Y^2+1)-(Y-c_1)\cdots (Y-c_q).
    \]
    By Hensel's lemma \cite[p.~129]{Neukirch}, $f$ splits into a product of degree
    one factors in $\F[[t^{-1}]][Y]$. Therefore, $f$ has $q+1$ distinct roots in 
    $\F((t^{-1}))$. So 
    $k$ is split at $\infty$, and hence $g_{\min}(q+1,q)=0$.
\end{proof}

\begin{comment}
\begin{thm}[\textbf {Vladimir Drinfeld Bound}]\label{thm; Drinfeld Bound}
    We have
    \[A(q)=\limsup_{g\rightarrow\infty}\frac{N_{q}(g)}{g}\leq {\sqrt{q}-1},\]
where $N_{q}(g)$ denotes the maximum number of rational places a global function field $F/\F$ of genus $g$ can have. 
\end{thm}
Denote 
\[B(q)=\liminf_{n\rightarrow\infty}\frac{g_{\min}(n,q)}{n}\]
lt follows from Serre bound that $B(q)\ge 1/\lfloor 2\sqrt{q}\rfloor$, and an improvement on this lower bound can be obtained by Theorem \ref{thm; Drinfeld Bound}.

\begin{thm} We have
    \[B(q)\ge \frac{1}{\sqrt{q}-1}.\]
\end{thm}

\begin{proof}
    By Proposition \ref{thm;ineq genus}, $g_{min}(n,q)\rightarrow \infty$ when $n\rightarrow \infty$. By Theorem \ref{thm; Drinfeld Bound}, for every $\epsilon>0$, we have 
    \[n\leq N_q(g_{\min}(n,q))\leq g_{\min}(n,q)(\sqrt{q}-1+\epsilon)\]
    for sufficiently large $n$. The result is followed by the arbitrary $\epsilon$.
\end{proof}

\end{comment}

%  Next, we will explicitly calculate $g_{\min}(n,q)$ for the low genus.

Suppose   $k/\F(t)$ is a degree n extension split at $\infty$ and $g_{k}=1$. By \cite[Theorem 4.26]{salvador2006topics}, $k$ is an elliptic function field, i.e.~$k$ is the function field of an elliptic curve $E$.  We denote  an elliptic  curve $E$ defined over $\F$ by  $E/\F$.
   Let  $\sharp E(\F)$ be  the number  of $\F$-rational points on $E$.

\begin{lem}\label{lem;g=1}
    Let $E/\F$ be an elliptic curve, and $k$ be the function field of $E$. Suppose $\sharp E(\F)\ge n$ and $q>2$. Then there exists an element $t\in k^{*}$ such that $[k:\F(t)]=n$ and the extension $k/\F(t)$ is split at $\infty$. 
\end{lem}
\begin{proof}
     Since $\sharp E(\F)\ge n$, there exist $n$ degree one divisors, say $D_{v_i}\  (1\leq i\leq n)$.
     We claim that there exists  $t\in k^{*}$ such that 
\begin{align}\label{eq;split}
(t)_{k,\infty}=\sum_{i=1}^nD_{v_i}.
\end{align}
Then we have  $[k:\F(t)]=n$ by Proposition \ref{prop,divisor}.  Moreover,
(\ref{eq;split}) also implies the
extension $k/\F(t)$ is split at $\infty$.

We prove the claim  by induction.
  By Lemma \ref{lem;rr}, we have 
     \begin{align}
     l(D_{v_i}+D_{v_j})&=\textup{deg}(D_{v_i}+D_{v_j})-g_k+1=2 \quad\mbox{and}\quad 
     l(D_{v_i})=1. \notag
     \end{align}
% Since $l(D_{v_i})=l(D_{v_j})=1$, 
So there exists $x_{ij}\in k^{*}$ such that
 \begin{align}\label{eq;xij}
  (x_{ij})_{k,\infty}=D_{v_i}+D_{v_j}.   
 \end{align}
 Suppose for $2\le  m<n$ there exists $t_{m}\in k^{*}$ such that
\[(t_m)_{k,\infty}=\sum_{i=1}^{m}D_{v_i}.\]
Note that 
\[(x_{1{(m+1)}})_{k,\infty}=D_{v_1}+D_{v_{m+1}}.\]
Suppose $k_{v_1}=\F((z))$ and $v_1(z)=1$.
Then the Laurent series  expansions of 
\begin{align}
x_{1(m+1)}&=\frac{a_{-1}}{z}+\sum_{i=0}^{\infty}a_i z^{i},\notag\\
t_m&=\frac{b_{-1}}{z}+\sum_{i=0}^{\infty}b_i z^{i}.\notag
\end{align}
 Since $q>2$, we can always find a constant $c\in \F^{*}$ such that $ca_{-1}+b_{-1}\neq 0$. Then for $t_{m+1}=cx_{1(m+1)}+t_m$ we have 
\[(t_{m+1})_{k,\infty}=(cx_{1(m+1)}+t_m)_{k,\infty}=\sum_{i=1}^{m+1}D_{v_i}.\]
 
\end{proof}

By this lemma, the question now is to find the maximal possible $\sharp E(\F)$. 
 Let 
\[A_q(1)=\sup\{\sharp N(E): \text{$E/\F$ curve of genus $1$}\}. \]
The following result is from Waterhouse \cite{waterhouse1969abelian} using the theorem from Tate and Honda \cite{honda1968isogeny}. The general problem of rational points on curves over finite fields can be found in  \cite{serre2020rational}.
\begin{lem}\label{thm;count}\cite[Theorem 2.6.3]{serre2020rational} 
\begin{align*}
    A_q(1)= \begin{cases}
   % 4 & \text{if } q=2,\\
    q+\lfloor2\sqrt{q}\rfloor & \text{where  $p|\lfloor2\sqrt{q}\rfloor$, $e$ is odd  and $e \ge 5$ }, \\
    q+\lfloor2\sqrt{q} \rfloor +1& \text{otherwise}.\end{cases}
    \end{align*}
\end{lem}
\begin{rem}\label{rem;ell}
    The elliptic curves over $\F$ with the same number of $\F$-rational points are isogenous \cite[Theorem 4.1]{waterhouse1969abelian}.
\end{rem}

\begin{prop}\label{thm;q>2}
  Let $N_q$ be as in Theorem \ref{thm;main} and      $q>2$. Then  $g_{min}(n,q)=1$ if and only if $q+2\leq n\leq N_{q}$. 
\end{prop}
\begin{proof}
It follows from Lemmas \ref{thm;ineq genus}  and \ref{thm;count}.
\end{proof}
%By Lemma \ref{lem;g=1}, we need to examine the case $q=2$.

\begin{prop}\label{thm;q=2}
 Suppose $q=2$. Then  $g_{\min}(n,q)=1$ if and only if  $n=N_q=4$.
\end{prop}
\begin{proof}
    According to Lemma \ref{thm;ineq genus}, we have $g_{\min}(n,2)=0$ for 
    $2\le n\le 3$ and  $g_{\min}(n,2)>1$ for $n>5$. 
    So we only need to consider the cases for  $n=4, 5$ where   $g_{\min}(n,2)\ge 1$ by Lemma \ref{thm;count}.

Let $E/\F$ be an elliptic curve with $\sharp E(\F)= 5$  and let  $k$ be its rational function field. 
Let $D_{v_i}\ (1\leq i\leq 5)$ be the divisors of degree one in $\mathcal{P}_k$. 
The proof of Lemma \ref{lem;g=1} implies that there exists
$x_{ij}\in k^{*}$ such that (\ref{eq;xij}) holds.
Let  $t_4=x_{12}+x_{34}$, then
\[(t_{4})_{k,\infty}=\sum_{i=1}^{4}D_{v_i}.\]
It follows that $k/\F(t_4)$ is a degree $4$ 
 extension  split at $\infty$. So $g_{\min}(4,2)=1$.
 
 We prove $g_{\min}(5,2)>1$ by contradiction. 
 Suppose $g_{\min}(5,2)=1$, then there exists a degree $5$ extension of $\F(t)$  split at $\infty$ with genus $1$. This extension must be isomorphic to the function field $k$ which corresponds to an elliptic curve $E$ with $\sharp E(\F)=5$. 
 This means that there exists
$t \in k^{*}$ such that 
\begin{align}
    \label{eq;t}
    (t )_{k,\infty}=\sum_{i=1}^{5}D_{v_i}.
\end{align}

By Corollary \ref{lem;rr}, we have 
$l(\sum_{i=1}^{5}D_{v_i})=5$. Since $L(D_{v_1}+D_{v_i})\subseteq L(\sum_{i=1}^{5}D_{v_i})$
for $2\le i\le 5$, we have 
 $x_{1i}$ belongs to the latter.   Thus 
\[t =c_1+\sum_{i=2}^{5}c_ix_{1i}\]
for some constants $c_i\in \F=\{0,1\}$. Note that  $v_i(t)=-1$ for $2\leq i\leq 5$ in view of (\ref{eq;t}), so $c_i=1$ for all $i$. Let $z\in k^{*}$ be a uniformizer at $v_1$ such that $v_1(z)=1$. Consider the Laurent expansion of $x_{1i}$ in $k_{v_1}=\F((z))$ and write
\[x_{1i}=\frac{1}{z}+\sum_{j=0}^{\infty}d_{ij}z^{j}.\]
Then 
\[v_1(t )=v_1(c_1+\sum_{i=2}^{5}x_{2i})
=v_1(c_1+\sum_{i=2}^{5}\sum_{j=0}^{\infty}d_{ij}z^{j})\ge 0.\]
This contradicts to $v_1(t )=-1$. Thus $g_{\min}(5,2)>1$.
\end{proof}

\begin{proof}
    [Proof of Theorem \ref{thm;algebraic}]
    It follows from (\ref{eq;use}), Propositions \ref{thm;q>2} and \ref{thm;q=2}. 
\end{proof}

%\subsection*{Financial Support} 
%{R.\ S.\ is supported by National Key Research and Development Program of China 2021YFA1003204, NSFC 12161141014 and NSF Shanghai 22ZR1406200.}

\bibliographystyle{plain}
\bibliography{refs}

\end{document}